\documentclass[10pt,reqno]{amsart}

\usepackage{amsthm,amsmath,amstext,amssymb,amscd,euscript, mathrsfs, dsfont,multicol,times,enumerate,subfig,sidecap}
\usepackage{enumerate}
\usepackage{amsmath, amssymb, amsthm}
\usepackage{mathrsfs}
\usepackage{esint}
\usepackage{xcolor}
\usepackage{mathtools}
\usepackage{bm}
\usepackage{bbm}

\usepackage[colorlinks=true, pdfstartview=FitV, linkcolor=blue, citecolor=red, urlcolor=blue]{hyperref} 

\newtheorem{thm}{Theorem}[section]

\newtheorem{lem}[thm]{Lemma}

\theoremstyle{definition}
\newtheorem{defn}[thm]{Definition}

\newtheorem{rem}[thm]{Remark}

\newcommand\bN{\mathbb{N}}

\newcommand\bQ{\mathbb{Q}}
\newcommand\bR{\mathbb{R}}
\newcommand\bS{\mathbb{S}}

\newcommand\bZ{\mathbb{Z}}

\newcommand\fR{\mathbb{R}}

\newcommand\cF{\mathcal{F}}
\newcommand\cG{\mathcal{G}}

\newcommand\cI{\mathcal{I}}

\newcommand\cL{\mathcal{L}}

\newcommand\cS{\mathcal{S}}
\newcommand\cM{\mathcal{M}}

\newcommand\cO{\mathcal{O}}

\newcommand\cQ{\mathcal{Q}}

\DeclareMathOperator*{\esssup}{ess\,sup}

\newcommand{\aint}{-\hspace{-0.40cm}\int}

\newcommand{\mysection}[1]{\section{#1}
\setcounter{equation}{0}}

\def\XXint#1#2#3{{\setbox0=\hbox{$#1{#2#3}{\int}$ }
\vcenter{\hbox{$#2#3$ }}\kern-.58\wd0}}

\makeatletter
\def\@tocline#1#2#3#4#5#6#7{\relax
  \ifnum #1>\c@tocdepth 
  \else
    \par \addpenalty\@secpenalty\addvspace{#2}%
    \begingroup \hyphenpenalty\@M
    \@ifempty{#4}{%
      \@tempdima\csname r@tocindent\number#1\endcsname\relax
    }{%
      \@tempdima#4\relax
    }%
    \parindent\z@ \leftskip#3\relax \advance\leftskip\@tempdima\relax
    \rightskip\@pnumwidth plus4em \parfillskip-\@pnumwidth
    #5\leavevmode\hskip-\@tempdima
      \ifcase #1
       \or\or \hskip 1em \or \hskip 2em \else \hskip 3em \fi%
      #6\nobreak\relax
    \dotfill\hbox to\@pnumwidth{\@tocpagenum{#7}}\par
    \nobreak
    \endgroup
  \fi}
\makeatother

\begin{document}
\title[A regularity theory for evolution equations in weighted mixed-norm Sobolev-Lipschitz spaces]
{A regularity theory for evolution equations with time-measurable pseudo-differential operators in weighted mixed-norm Sobolev-Lipschitz spaces}
\author[J.-H. Choi]{Jae-Hwan Choi}
\address[J.-H. Choi]{Research Institute of Mathematics, Seoul National University, 1 Gwanak-ro, Gwanak-gu, Seoul 08826, Republic of Korea}
\email{choijh1223@snu.ac.kr}




\maketitle
\begin{abstract}
This paper investigates the existence, uniqueness, and regularity of solutions to evolution equations with time-measurable pseudo-differential operators in weighted mixed-norm Sobolev-Lipschitz spaces.
We also explore trace embedding and continuity of solutions.
\end{abstract}
\tableofcontents

\mysection{Introduction and Main results}
\subsection{Goal and Setting} Consider the following Cauchy problem:
\begin{equation}
\label{24.01.11.10.10}
        \begin{cases}
        \partial_tu(t,x)=\psi(t,-i\nabla)u(t,x)+f(t,x),\quad &(t,x)\in(0,T)\times\bR^d,\\
        u(0,x)=u_0(x)\quad &x\in\bR^d.
    \end{cases}
\end{equation}
Here, $\psi(t, -i\nabla)$ represents a pseudo-differential operator defined by
\begin{align*}
    \psi(t,-i\nabla)u(t,x)&:=\frac{1}{(2\pi)^{d/2}}\int_{\bR^d}\psi(t,\xi)\cF[u(t,\cdot)](\xi)\mathrm{e}^{ix\cdot\xi}\mathrm{d}\xi\\
    &=\frac{1}{(2\pi)^d}\int_{\bR^d}\int_{\bR^d}\psi(t,\xi)u(t,y)\mathrm{e}^{i(x-y)\cdot\xi}\mathrm{d}y\mathrm{d}\xi,
\end{align*}
where $\mathcal{F}$ denotes the Fourier transform with respect to $x$-variables.
The symbol $\psi$ belongs to a class $\mathbb{S}(\gamma,\kappa,M)$, which includes symbols of second-order elliptic operators with time-measurable coefficients:
$$
\psi(t,\xi)=-\sum_{i,j=1}^{d}a^{ij}(t)\xi^i\xi^j\Longleftrightarrow \psi(t,-i\nabla)=\sum_{i,j=1}^{d}a^{ij}(t)D_{x^ix^j},
$$
where $\xi=(\xi^1,\cdots,\xi^d)\in\bR^d$ and $A(t):=[a^{ij}(t)]_{d\times d}$ is a real $d\times d$ matrix satisfying ellipticity condition. 
\begin{defn}
\label{24.01.11.09.20}
    For $M\geq\kappa>0$ and $\gamma>0$, the set $\mathbb{S}(\gamma,\kappa,M)$ is a collection of complex-valued measurable function $\psi(t,\xi)$ satisfying
    $$
    \Re[-\psi(t,\xi)]\geq\kappa|\xi|^{\gamma},\quad \forall (t,\xi)\in\bR\times\bR^d,
    $$
    and
    $$
|D^{\alpha}_{\xi}\psi(t,\xi)|\leq M|\xi|^{\gamma-|\alpha|},\quad \forall (t,\xi)\in\bR\times(\bR^{d}\setminus\{0\}),
$$
for any ($d$-dimensional) multi-index $\alpha$ with $|\alpha| \leq \lfloor d/2\rfloor+1$.
\end{defn}

In this paper, we investigate the existence, uniqueness, and regularity of solutions to \eqref{24.01.11.10.10} in weighted mixed norm Sobolev-Lipschitz space $L_p((0,T),w\,\mathrm{d}t;\Lambda^{\phi}(\bR^d))$:
\begin{align*}
    \|u\|_{L_p((0,T),w\,\mathrm{d}t;\Lambda^{\phi}(\bR^d))}:=\begin{cases}
        \left(\int_0^T\|u(t,\cdot)\|_{\Lambda^{\phi}(\bR^d)}^pw(t)\mathrm{d}t\right)^{1/p},\quad &p\in[1,\infty)\\
    \esssup\limits_{t\in(0,T)}\|u(t,\cdot)\|_{\Lambda^{\phi}(\bR^d)},\quad &p=\infty.
    \end{cases}
\end{align*}
In this context, $\Lambda^{\phi}(\mathbb{R}^d)$ denotes a Lipschitz space of variable order $\phi$.
Additionally, the weight function $w$ belongs to the Muckenhoupt $A_p(\mathbb{R})$ class.

\begin{defn}[Lipschitz space of variable order $\phi$]
\label{24.01.14.12.34}

$(i)$ For a function $\phi:\bR_+\rightarrow \bR_+$, we define 
$$
s_\phi(\lambda) := \sup_{t>0}\frac{\phi(\lambda t)}{\phi(t)}\,,
$$
such that $s_{\phi}:\mathbb{R}+\rightarrow (0,\infty]$.
For $a, b \in \mathbb{R}$,
\begin{align}\label{ftn class}
\mathcal{I}(a, b) := \{\phi : \text{$s_\phi(\lambda) = O(\lambda^a)$ as $\lambda \downarrow 0$, and $s_\phi(\lambda) = O(\lambda^b)$ as $\lambda \uparrow \infty$} \}.
\end{align}
We also define $\mathcal{I}_o(a,b)$, where $o$ indicates the replacement of the big $O$ notation with little $o$ notation in \eqref{ftn class}.

$(ii)$ Let $p\in[1,\infty]$, $L>0$, and $\phi\in\cI_o(0,L)$.
    The space $\Lambda_p^{\phi}(\bR^d)$ is defined by
    $$
    \Lambda_p^{\phi}(\bR^d):=\{f\in L_{\infty}(\bR^d):\|f\|_{\Lambda_p^{\phi}(\bR^d)}<\infty\},
    $$
    where the norm $\|f\|_{\Lambda_p^{\phi}(\bR^d)}$ is given by
    \begin{eqnarray*}
    \|f\|_{\Lambda_p^{\phi}(\bR^d)}:=\|f\|_{L_{\infty}(\bR^d)}+
    \begin{cases}
        \left(\int_{\bR^d}\phi(|h|^{-1})^p\|\mathcal{D}_{h}^{L_0}f\|_{L_{\infty}(\bR^d)}^p\frac{\mathrm{d}h}{|h|^d}\right)^{1/p},\quad &p\in[1,\infty),\\
        \sup\limits_{h\in\bR^d}\phi(|h|^{-1})\|\mathcal{D}_{h}^{L_0}f\|_{L_{\infty}(\bR^d)},\quad &p=\infty.
        \end{cases}
    \end{eqnarray*}
Here $\mathcal{D}_hf(x):=f(x+h)-f(x)$, $\mathcal{D}_h^nf(x)=\mathcal{D}_h^{n-1}(\mathcal{D}_hf)(x)$ and $L_0$ is the smallest integer which is greater than or equal
to $L$.
We denote $\Lambda_{\infty}^{\phi}(\bR^d)=\Lambda^{\phi}(\bR^d)$.
\end{defn}

\begin{rem}
\label{24.01.14.12.36}
It can be easily checked that $\Lambda_{p}^{\phi}(\bR^d)$ encompasses non-separable Banach spaces, such as the classical H\"older space.
For example, if we choose $\phi(\lambda) := \lambda^{\alpha}$ (for $\alpha \in (0,1)$), belonging to $\mathcal{I}_o(0,1)$, then $L_0$ is set to 1.
Under these conditions, $\Lambda^{\phi}(\mathbb{R}^d)$ is equivalent to the H\"older space of order $\alpha$, denoted as $C^{\alpha}(\mathbb{R}^d)$.
Additional properties and characteristics of the function classes $\mathcal{I}(a,b)$ and $\mathcal{I}_o(a,b)$ are detailed in \cite{CLSW2023trace,GP1977}.
\end{rem}

\begin{defn}[Muckenhoupt's class]
\label{24.01.11.09.18}
For $p\in(1,\infty)$, let $A_p=A_p(\bR^d)$ be the class of all nonnegative and locally integrable functions $w$ satisfying
\begin{align}
\label{def ap}
[w]_{A_p}:=\sup_{\cQ\,\text{cubes in }\bR^d}\left(\aint_{\cQ}w(x)\mathrm{d}x\right)\left(\aint_{\cQ}w(x)^{-1/(p-1)}\mathrm{d}x\right)^{p-1}<\infty.
\end{align}
The class of $A_1=A_1(\bR^d)$ is defined as a collection of nonnegative and locally integrable functions $w$ satisfying
$$
\cM w(x)\leq Cw(x),\quad \forall x\in\bR^d,
$$
where $\cM$ is the Hardy-Littlewood maximal operator and $C$ is independent of $x$.
The class $A_{\infty}=A_{\infty}(\bR^d)$ could be defined as the union of $A_p(\bR^d)$ for all $p\in[1,\infty)$, \textit{i.e.}
    $$
    A_\infty(\bR^d)=\bigcup_{p\in[1,\infty)}A_p(\bR^d).
    $$
\end{defn}

\subsection{Background and Main results}
Firstly, we revisit some historical results on the regularity estimates of second-order parabolic partial differential equations (PDEs) within the mixed norm space $L_p(\mathbb{R};C^{2+\alpha}(\mathbb{R}^d))$.
The equation of interest is as follows:
\begin{equation}
\label{24.01.12.09.24}
    \partial_tu(t,x)=\sum_{i,j=1}^da^{ij}(t)D_{x^ix^j}u(t,x)+f(t,x),\quad (t,x)\in\bR\times\bR^d.
\end{equation}

The case for $p = \infty$, known as the \emph{partial Schauder estimates}, gained prominence following Brandt's work \cite{Brandt1969}.
In \cite{Brandt1969}, Brandt demonstrated that, unlike the \emph{classical Schauder estimates} (refer to \textit{e.g.}, \cite{krylov1996lectures,Lieberman1996}), the partial Schauder estimates are valid even if $a^{ij}$ in \eqref{24.01.12.09.24} are merely measurable.
Additional references and works on partial Schauder estimates can be found in \cite{Dong_Kim2011,Dong_Kim2019,Knerr1980,Krylov_2010,Lieberman1992,Lorenzi2000}.

To the best of author's knowledge, the case for $p \neq \infty$ was initially introduced by Krylov in \cite{Krylov_2002,Krylov_2003}. The primary concepts in Krylov's studies are summarized as follows:
\begin{enumerate}
\item The main idea of \cite{Krylov_2002} was the development of a systematic approach, as laid out in \cite{Krylov_2002_1}, using the vector-valued Calder\'on-Zygmund theorem. For further generalizations of \cite{Krylov_2002_1}, refer to \cite{Dong_Kim2015,Dong_Kim2018,Krylov_2002_1,Krylov_2021,PST2017}.
\item The focus of \cite{Krylov_2003} was refining Brandt's estimates (from \cite{Brandt1969}) using simple heat potential estimates and well-established properties of the Hardy-Littlewood maximal function.
\end{enumerate}
These approaches led to significant generalizations, as seen in \cite{Kim2018,ST2021}.
Stinga and Torrea (\cite{ST2021}) expanded upon Krylov's methodology in \cite{Krylov_2002}, proving regularity estimates of solutions to \eqref{24.01.12.09.24} in $L_p(\mathbb{R},w\,\mathrm{d}t;C^{2+\alpha}(\mathbb{R}^d))$, where $w$ is a weight function in the Muckenhoupt $A_p(\mathbb{R})$ class.
Kim (\cite{Kim2018}) proposed regularity estimates of solutions to \eqref{24.01.11.10.10} with zero initial data in $L_p((0,T);\Lambda^{\gamma+\alpha}(\mathbb{R}^d))$, where $\Lambda^{\gamma+\alpha}(\mathbb{R}^d)$ is a Lipschitz space of order $\gamma+\alpha$ (see Definition \ref{24.01.14.12.34} and Remark \ref{24.01.14.12.36}).
Kim's main idea in \cite{Kim2018} involved an interpolation between two estimates: a partial Schauder estimate and a weak $L_1$ estimate.
For these estimates, refined calculations of the fundamental solution (see Definition \ref{24.01.14.12.42}) for \eqref{24.01.11.10.10} were necessary.

Despite the advancements in the study of evolution equations within mixed norm spaces such as $L_p(\mathbb{R};C^{2+\alpha}(\mathbb{R}^d))$ and $L_p((0,T);\Lambda^{\gamma+\alpha}(\mathbb{R}^d))$, as evidenced by the aforementioned results, certain fundamental and intriguing questions remain unaddressed.
These questions include:
\begin{enumerate}
    \item \textbf{Weighted Regularity and Variable Order:} Apart from Stinga and Torrea's results, all previous studies have focused on non-weighted regularity results in spaces of constant order, such as $C^{2+\alpha}(\mathbb{R}^d)$.
    This raises the question: \emph{Is it feasible to achieve regularity estimates for \eqref{24.01.11.10.10} within the weighted mixed norm space $L_p((0,T),w\,\mathrm{d}t;\Lambda^{\phi}(\mathbb{R}^d))$?}
    \item \textbf{Initial Data and Trace in $L_p((0,T),w\,\mathrm{d}t;\Lambda^{\phi}(\mathbb{R}^d))$:} None of the studies mentioned earlier, such as \cite{Kim2018,Krylov_2002,Krylov_2003,ST2021}, consider the initial data.
    This leads to the inquiry: \emph{Are the findings of \cite{Kim2018,Krylov_2002,Krylov_2003,ST2021} still applicable in $L_p((0,T),w\,\mathrm{d}t;\Lambda^{\phi}(\mathbb{R}^d))$ when initial data is factored in? If so, what is the optimal initial data space?}
    \item \textbf{Continuity of the Solution:} When a solution is uniquely determined in our target space $L_p((0,T),w\,\mathrm{d}t;\Lambda^{\phi}(\mathbb{R}^d))$, \emph{Can we ascertain H\"older or Zygmund type regularity for these solutions, such as in $C^{\alpha,\beta}_{t,x}$?}
    \item \textbf{General operators and Strong Solution:} Excluding Kim's research \cite{Kim2018}, the predominant focus has been on second-order parabolic PDEs, rendering the discussion about the type of solution (mild, weak, or strong) less pertinent. However, Kim's study \cite{Kim2018}, due to the non-local nature of the main operator $\psi(t, -i\nabla)$, addressed only the regularity estimates of the \emph{weak} solution. This raises a crucial question: \emph{Is it possible to discuss a strong solution for \eqref{24.01.11.10.10}?}
    \end{enumerate}

This paper is dedicated to addressing the questions previously outlined.
To achieve this, we have developed a unified approach that synergizes the methodologies of Krylov, as seen in \cite{Krylov_2003}, with those of Kim from \cite{Kim2018}.
Krylov's contribution, detailed in \cite{Krylov_2003}, involves proving the following point-wise estimates using straightforward heat potential calculations,
\begin{equation}
\label{24.01.15.17.42}
[D^2_xu(t,\cdot)]_{C^{\alpha}(\bR^d)}\leq N\sup_{r>0}\frac{1}{2r}\int_{t-r}^{t+r}\left[\partial_su(s,\cdot)-\sum_{i,j=1}^da^{ij}(s)D_{x^ix^j}u(s,\cdot)\right]_{C^{\alpha}(\bR^d)}\mathrm{d}s,
\end{equation}
where $[f(t,\cdot)]_{C^{\alpha}(\bR^d)}:=\sup_{x\neq y}\frac{|f(t,x)-f(t,y)|}{|x-y|^{\alpha}}$.
In \eqref{24.01.15.17.42}, the right-hand side represents the Hardy-Littlewood maximal operator:
$$
\mathcal{M}f(t):=\sup_{r>0}\frac{1}{2r}\int_{t-r}^{t+r}|f(s)|\mathrm{d}s.
$$
Given that $\mathcal{M}:L_p(\mathbb{R}) \to L_p(\mathbb{R})$ is bounded for $p \in (1,\infty)$ (see, for example, \cite[Theorem 2.1.6]{grafakos2014classical}), Krylov's a priori estimates are thus obtained.
In \cite{Kim2018}, Kim proved the partial Schauder estimates and weak $L_1$ estimates by leveraging the fundamental solution and the Littlewood-Paley characterization of the Lipschitz space. The following estimates were established:
\begin{equation*}
    \begin{gathered}
        \|u\|_{L_{\infty}((0,T);\Lambda^{\gamma+\alpha}(\bR^d))}\leq N\|\partial_tu-\psi(\cdot,-i\nabla)u\|_{L_{\infty}((0,T);\Lambda^{\alpha}(\bR^d))},\\
        \|u\|_{L_{1,\infty}((0,T);\Lambda^{\gamma+\alpha}(\bR^d))}\leq N\|\partial_tu-\psi(\cdot,-i\nabla)u\|_{L_{1}((0,T);\Lambda^{\alpha}(\bR^d))}.
    \end{gathered}
\end{equation*}
The Marcinkiewicz interpolation theorem gives Kim's a priori estimates.
Inspired by these results, our strategy is twofold:
\begin{itemize}
    \item Deriving estimates on the fundamental solution (Theorem \ref{FSetimate}).
    \item Developing point-wise estimates in the vein of \eqref{24.01.15.17.42}, utilizing the Littlewood-Paley characterization (Lemma \ref{23.12.05.14.49}).
\end{itemize}
For addressing the initial data and trace, we primarily adopt the approach delineated in \cite{Choi_Kim_Lee2023,CLSW2023trace}, which elaborates on results concerning the initial data and trace in the space $L_p(\mathbb{R},w_1\,\mathrm{d}t;L_q(\mathbb{R}^d,w_2\,\mathrm{d}x))$.

Here are the main results of this paper.
The proofs will be given in Section \ref{24.02.03.18.23} and \ref{24.02.03.18.24}.
\begin{defn}[Solution space]
For $\gamma\in(0,\infty)$, $L>0$, $\phi\in\cI_o(0,L)$, $p\in[1,\infty]$ and $w\in A_p(\bR)$, denote
$$
\phi_{\gamma}(\lambda):=\phi(\lambda)\lambda^{\gamma},\quad \phi_{\gamma,p,w}(\lambda):=\phi(\lambda)\lambda^{\gamma}\left(\int_0^{\lambda^{-\gamma}}w(s)\mathrm{d}s\right)^{1/p}.
$$
    The space $\mathbf{H}_{p,w}^{\phi,{\gamma}}((0,T)\times\bR^d)$ is a collection of $u\in L_p((0,T),w\,\mathrm{d}t;\Lambda^{\phi_{\gamma}}(\bR^d))$ which satisfies that there exist $u_0\in \Lambda_p^{\phi_{\gamma,p,w}}(\bR^d)$ and $f\in L_p((0,T),w\,\mathrm{d}t;\Lambda^{\phi}(\bR^d))$ such that
    $$
    u(t,x)=u_0(x)+\int_0^tf(s,x)\mathrm{d}s
    $$
    for all $(t,x)\in [0,T)\times\bR^d$.
    We also define
    $$
    \|u\|_{\mathbf{H}_{p,w}^{\phi,{\gamma}}((0,T)\times\bR^d)}:=\|u\|_{L_p((0,T),w\,\mathrm{d}t;\Lambda^{\phi_{\gamma}}(\bR^d))}+\|u_0\|_{\Lambda_p^{\phi_{\gamma,p,w}}(\bR^d)}+\|f\|_{L_p((0,T),w\,\mathrm{d}t;\Lambda^{\phi}(\bR^d))}.
    $$
\end{defn}
\begin{thm}[Existence, Uniqueness, and Regularity of solutions]
\label{23.12.05.09.53}
    Let $p\in(1,\infty]$, $w\in A_p(\bR)$, $\psi\in\mathbb{S}(\gamma,\kappa,M)$ and $\phi\in\cI_o(0,L)$.
   Then for $u_0\in \Lambda_{p}^{\phi_{\gamma,p,w}}(\bR^d)$ and $f\in L_{p}((0,T),w\,\mathrm{d}t;\Lambda^{\phi}(\bR^d))$, there exists a unique strong solution $u\in\mathbf{H}_{p,w}^{\phi,{\gamma}}((0,T)\times\bR^d)$ to
    \begin{equation}
    \label{23.11.27.20.31}
        \begin{cases}
        \partial_tu(t,x)=\psi(t,-i\nabla)u(t,x)+f(t,x),\quad (t,x)\in(0,T)\times\bR^d,\\
        u(0,x)=u_0(x),
    \end{cases}
    \end{equation}
    with the estimate
    \begin{align}
        \label{23.12.05.12.28}
        \|u\|_{\mathbf{H}_{p,w}^{\phi,{\gamma}}((0,T)\times\bR^d)}\leq N\left(\|u_0\|_{{\Lambda_p^{\phi_{\gamma,p,w}}}(\bR^d)}+\|f\|_{L_{p}((0,T),w\,\mathrm{d}t;\Lambda^{\phi}(\bR^d))}\right),
    \end{align}
    where $N=N(d,\delta,\gamma,\kappa,M,p,T,[w]_{A_p(\bR)})$.
\end{thm}
\begin{rem}
Theorem \ref{23.12.05.09.53} partially holds for the case $p=1$:
\begin{enumerate}[(i)]
    \item If $f\equiv 0$, then Theorem \ref{23.12.05.09.53} is valid for the case $p=1$.
    \item Theorem \ref{23.12.05.09.53} holds when the solution $u$ is in the weak space $L_{1,\infty}((0,T),w\,\mathrm{d}t;\Lambda^{\phi_{\gamma}}(\mathbb{R}^d))$, rather than in $L_{1}((0,T),w\,\mathrm{d}t;\Lambda^{\phi_{\gamma}}(\mathbb{R}^d))$.
    Here, the norm of $L_{1,\infty}((0,T),w\,\mathrm{d}t;\Lambda^{\phi_{\gamma}}(\mathbb{R}^d))$ is defined as
    $$
    \|u\|_{L_{1,\infty}((0,T),w\,\mathrm{d}t;\Lambda^{\phi_{\gamma}}(\mathbb{R}^d))}:=\sup_{\lambda>0}\lambda w(\{t\in(0,T):\|u(t,\cdot)\|_{\Lambda^{\phi_{\gamma}}(\mathbb{R}^d)}\geq\lambda\}),
    $$
    where $w(A):=\int_{\bR}1_{A}(t)w(t)\mathrm{d}t$.
\end{enumerate}
Corresponding estimates for cases (i) and (ii) will be detailed in Theorem \ref{23.11.26.16.09}.
\end{rem}

\begin{thm}[Embedding theorems]
\label{24.02.03.17.37}
Let $p\in(1,\infty]$, $w\in A_p(\bR)$, $\phi\in\cI_o(0,L)$ and $u\in\mathbf{H}_{p,w}^{\phi,{\gamma}}((0,T)\times\bR^d)$.
\begin{enumerate}[(i)]
    \item \textbf{(Trace embedding)} Then there exists a constant $N=N(p,[w]_{A_p(\bR)})$ such that,
    \begin{equation}
    \label{24.01.15.23.11}
    \sup_{0\leq t\leq T}\|u(t,\cdot)\|_{\Lambda^{\phi_{\gamma,p,w}}(\bR^d)}\leq \sup_{0\leq t\leq T}\|u(t,\cdot)\|_{\Lambda_p^{\phi_{\gamma,p,w}}(\bR^d)}\leq N\|u\|_{\mathbf{H}_{p,w}^{\phi,{\gamma}}((0,T)\times\bR^d)}.
    \end{equation}
    \item \textbf{(Continuity of solutions)} Then there exists a constant $N=N(L_0,p,[w]_{A_p(\bR)})$ such that,
    $$
    \|u\|_{C([0,T]\times\bR^d)}+\sup_{(s,h_1,h_2)\in D_T}\frac{\|\mathcal{D}_{(h_1/L_0,h_2)}^{L_0}u(s,\cdot)\|_{C(\bR^d)}}{\widetilde{W}(s+h_1,s)^{1-\frac{1}{p}}+\phi_{\gamma,p,w}(|h_2|^{-1})^{-1}}\leq N\|u\|_{\mathbf{H}_{p,w}^{\phi,{\gamma}}((0,T)\times\bR^d)},
    $$
    where $\widetilde{W}(t,s)^{1-\frac{1}{p}}:=\left(\int_s^tw^{-\frac{1}{p-1}}(r)\mathrm{d}r\right)^{1-\frac{1}{p}}$, $\mathcal{D}_{(h_1,h_2)}u(t,x)(=\mathcal{D}_{(h_1,h_2)}^{1}u(t,x)):=u(t+h_1,x+h_2)-u(t,x)$, $\mathcal{D}_{(h_1,h_2)}^{n+1}u(t,x):=\mathcal{D}_{(h_1,h_2)}^{n}(\mathcal{D}_{(h_1,h_2)}u)(t,x)$ and $D_T:=\{(s,h_1,h_2):0\leq s\leq s+h_1\leq T, h_2\in\bR^d\}$.
\end{enumerate}
\end{thm}

\begin{rem}
If $\gamma\in(0,1)$, $w(t):=1$ and $\phi(\lambda):=\lambda^{\alpha}$ for $\alpha\in(0,\gamma)$, then
$$
\widetilde{W}(s+h_1,s)^{1-\frac{1}{p}}\simeq h_1^{1-\frac{1}{p}},\quad \phi_{\gamma,p,w}(\lambda)\simeq \lambda^{\alpha+\gamma\left(1-\frac{1}{p}\right)}.
$$
This implies that
\begin{equation}
\label{24.01.20.11.14}
    C_{t,x}^{\widetilde{W}^{1-\frac{1}{p}},\phi_{\gamma,p,w}}([0,T]\times\bR^d)=C_{t,x}^{1-\frac{1}{p},\alpha+\gamma-\frac{\gamma}{p}}([0,T]\times\bR^d),
\end{equation}
where $C_{t,x}^{\widetilde{W}^{1-\frac{1}{p}},\phi_{\gamma,p,w}}([0,T]\times\bR^d)$ denotes a normed space with the norm defined as
    $$
    \|u\|_{C([0,T]\times\bR^d)}+\sup_{(s,h_1,h_2)\in D_T}\frac{\|\mathcal{D}_{(h_1/L_0,h_2)}^{L_0}u(s,\cdot)\|_{C(\bR^d)}}{\tilde{W}(s+h_1,s)^{1-\frac{1}{p}}+\phi_{\gamma,p,w}(|h_2|^{-1})^{-1}},
    $$
    and
    $C_{t,x}^{1-\frac{1}{p},\alpha+\gamma-\frac{\gamma}{p}}([0,T]\times\bR^d)$ denotes the anisotropic H\"older space with the norm defined as
    $$
    \|u\|_{C([0,T]\times\bR^d)}+\sup_{(t,x),(s,y)\in[0,T]\times\bR^d}\frac{|u(t,x)-u(s,y)|}{|t-s|^{1-\frac{1}{p}}+|x-y|^{\alpha+\gamma-\frac{\gamma}{p}}}.
    $$

For general $w$ and $\phi$, the equivalence such as \eqref{24.01.20.11.14} cannot be expected.
However, there is a relation between $C_{t,x}^{\widetilde{W}^{1-\frac{1}{p}},\phi_{\gamma,p,w}}([0,T]\times\bR^d)$ and the classical anisotropic H\"older spaces.
Drawing from \cite[Proposition 7.1.5-(8), Corollary 7.2.6, and Proposition 7.2.8]{grafakos2014classical} and \cite[Lemma 2.3]{CLSW2023trace}, there exists an $\varepsilon \in (0,1/2)$ such that the following inequalities hold:
    $$
    |h_1|^{1-\varepsilon}\lesssim\widetilde{W}(s+h_1,s)^{1-\frac{1}{p}}\lesssim |h_1|^{\varepsilon},\quad |h_2|^{L+\gamma-\varepsilon}\lesssim\phi_{\gamma,p,w}(|h_2|^{-1})^{-1}\lesssim |h_2|^{\varepsilon},
    $$
    where $(s,h_1,h_2)\in D_T$ and $|h_2|\leq 1$.
    Consequently, we can deduce that
    $$
    C_{t,x}^{\varepsilon,\varepsilon}([0,T]\times\bR^d)\subseteq C_{t,x}^{\widetilde{W}^{1-\frac{1}{p}},\phi_{\gamma,p,w}}([0,T]\times\bR^d)\subseteq C_{t,x}^{1-\varepsilon,L+\gamma-\varepsilon}([0,T]\times\bR^d).
    $$
\end{rem}

\subsection{Further research directions}
This study focused on evolution equations with time-measurable pseudo-differential operators independent of the spatial variable $x$.
However, several intriguing extensions remain unexplored:

\begin{enumerate}
    \item \textbf{$x$-dependent operators.}
    The results presented in this paper, as well as prior studies, focus on leading operators $\psi(t,-i\nabla)$ that are independent of the spatial variable $x$.
    Extending the theory to spatially dependent operators, such as second-order elliptic operators in either divergence or non-divergence form, is a natural and significant next step.
    For instance:
$$
\mathcal{L}=\sum_{i,j=1}^{d}a^{ij}(t,x)D_{x^ix^j},\quad \text{or}\quad \mathcal{L}=\sum_{i,j=1}^{d}D_{x^i}\left(a^{ij}(t,x)D_{x^j}\right).
$$
    \item \textbf{Variable-order non-local operators.}
    The current study assumes operators of constant order $\gamma$.
    Exploring variable-order non-local operators, such as the infinitesimal generators of subordinate Brownian motions, would be an interesting extension.
\end{enumerate}
I am preparing research to address these extensions, focusing on the challenges posed by spatially dependent operators and variable-order non-local operators.

\subsection{Notations}
Throughout this paper, we will use the following standard notations.

$\bullet$ $\bR$, $\bQ$, $\bZ$, and $\bN$ are the set of all real numbers, rational numbers, integers, and natural numbers, respectively.

$\bullet$ For a set $\{a, b, \ldots \}$ and $X, Y \in \bR$, we say $X \lesssim_{a, b, \ldots} Y$, if $X \le N Y$ holds for some $N = N(a, b, \ldots) > 0$. Sometimes we omit $a, b, \ldots$ and just say $X \lesssim Y$ if the dependency on $a, b, \ldots$ is clear from the context.

$\bullet$ 
        For given $p \in [1,\infty)$, a normed space $F$, and a  measure space $(X,\mathcal{M},\mu)$,  $L_{p}(X,\cM,\mu;F)$ denotes the space of all $\mathcal{M}^{\mu}$-measurable functions $u : X \to F$ with the norm 
        \[
            \left\Vert u\right\Vert _{L_{p}(X,\cM,\mu;F)}:=\left(\int_{X}\left\Vert u(x)\right\Vert _{F}^{p}\mu(\mathrm{d}x)\right)^{1/p}<\infty
        \]
        where $\mathcal{M}^{\mu}$ denotes the completion of $\cM$ with respect to the measure $\mu$. We also denote by $L_{\infty}(X,\cM,\mu;F)$ the space of all $\mathcal{M}^{\mu}$-measurable functions $u : X \to F$ with the norm
        $$
            \|u\|_{L_{\infty}(X,\cM,\mu;F)}:=\inf\left\{r\geq0 : \mu(\{x\in X:\|u(x)\|_F\geq r\})=0\right\}<\infty.
        $$
        If there is no confusion for the given measure and $\sigma$-algebra, we usually omit them.
        
$\bullet$  
        For $\cO\subset \fR^d$ and a normed space $F$, we denote by $C(\cO;F)$ the space of all $F$-valued continuous functions $u : \cO \to F$ with the norm 
        \[
            |u|_{C}:=\sup_{x\in O}|u(x)|_F<\infty.
        \]

$\bullet$ For $p \in [1, \infty]$, $p' := \frac{p}{p-1}$ is the H\"older conjugate ($\frac{1}{0} := \infty$ and $\frac{\infty}{\infty} := 1$).

$\bullet$ $1_A$ is a characteristic function on a set $A$.

\mysection{Existence, Uniqueness, and Regularity of solutions: Proof of Theorem \ref{23.12.05.09.53}}
\label{24.02.03.18.23}
In this section, we focus on proving Theorem \ref{23.12.05.09.53}.
We begin by introducing the Littlewood-Paley characterization of the space $\Lambda_p^{\phi}(\mathbb{R}^d)$.
We choose a function $\Psi$ from the Schwartz class $\mathcal{S}(\mathbb{R}^d)$, whose Fourier transform, $\mathcal{F}[\Psi]$, is nonnegative, supported within the annulus $\{\xi \in \bR^d : \frac{1}{2}\leq |\xi| \leq 2\}$.
We ensure that $\sum_{j\in\mathbb{Z}} \mathcal{F}[\Psi](2^{-j}\xi) = 1$ for all $\xi \not=0$.
Then, we define the Littlewood-Paley projection operators $\Delta_j$ and $S_0$ as 
$$
\Delta_jf(x):=(\Psi_j\ast f)(x),\quad S_0f(x):=(\Phi\ast f)(x),
$$
where $\Psi_j(x):=2^{jd}\Psi(2^jx)$ and $\Phi(x):=\sum_{j\leq 0}\Psi_j(x)$, respectively.
For $p \in [1, \infty]$, $L > 0$, and $\phi \in \mathcal{I}_o(0,L)$, the norm equivalence for $\|\cdot\|_{\Lambda^{\phi}_p(\mathbb{R}^d)}$ is established as follows, as presented in \cite{CLSW2023}:
\begin{equation}
\label{24.01.15.22.40}
        \|f\|_{\Lambda_p^{\phi}(\bR^d)}\simeq
    \|S_0f\|_{L_{\infty}(\bR^d)}+\begin{cases}
\left(\sum_{j=1}^{\infty}\phi(2^j)^p\|\Delta_jf\|_{L_{\infty}(\bR^d)}^p\right)^{1/p},\quad &p\in[1,\infty),\\
\sup_{j\in\bN}\phi(2^{j})\|\Delta_jf\|_{L_{\infty}(\bR^d)},\quad &p=\infty,
\end{cases}
\end{equation}
Henceforth, we will utilize the norm equivalence presented in \eqref{24.01.15.22.40} in place of Definition \ref{24.01.14.12.34}.

\subsection{Continuity of operators}
In general, functions within $\Lambda^{\phi_{\gamma}}(\mathbb{R}^d)$ are not integrable, as evidenced by non-zero constant functions.
Consequently, the expression $\psi(t,-i\nabla)f$ for any $f \in \Lambda^{\phi_{\gamma}}(\mathbb{R}^d)$ is defined in the sense of tempered distributions.
To discuss a strong solution, it is necessary to propose an alternative definition for the operator $\psi(t,-i\nabla)$.

Let us consider $\psi \in \mathbb{S}(\gamma,\kappa,M)$ and $f \in \Lambda^{\phi_{\gamma}}(\mathbb{R}^d)$. We define the following:
 \begin{equation}
 \label{23.11.30.15.20}
     \begin{gathered}
         \psi_j(t,-i\nabla)f(x):=\sum_{k=-1}^{1}\int_{\bR^d}\psi(t,-i\nabla)\Psi_{j+k}(y)\Delta_jf(x-y)\mathrm{d}y,\\
         \psi_0(t,-i\nabla)f(x):=\int_{\bR^d}\psi(t,-i\nabla)(\Phi+\Psi_1)(x)S_0f(x-y)\mathrm{d}y.
     \end{gathered}
 \end{equation}
and
\begin{equation}
\label{23.11.30.16.33}
    \tilde{\psi}(t,-i\nabla)f(x):=\psi_0(t,-i\nabla)f(x)+\sum_{j=1}^{\infty}\psi_j(t,-i\nabla)f(x).
\end{equation}
If $f \in \Lambda^{\phi_{\gamma}}(\mathbb{R}^d) \cap \mathcal{S}(\mathbb{R}^d)$, the almost orthogonality of $\Delta_j$ and $S_0$ implies that:
\begin{equation*}
    \tilde{\psi}(t,-i\nabla)f(x)=\cF^{-1}[\psi(t,\cdot)\cF[f]](x)=:\psi(t,-i\nabla)f(x).
\end{equation*}
Henceforth, the notation $\psi(t,-i\nabla)f$ for any $f \in \Lambda^{\phi_{\gamma}}(\mathbb{R}^d)$ denotes $\tilde{\psi}(t,-i\nabla)f$.
\begin{thm}
\label{23.11.14.14.54}
    Let $\psi\in\mathbb{S}(\gamma,\kappa,M)$ and $f\in\Lambda^{\phi_{\gamma}}(\bR^d)$.
    \begin{enumerate}[(i)]
        \item Then for any $\varphi\in\cS(\bR^d)$,
        \begin{align*}
            \int_{\bR^d}\psi(t,-i\nabla)f(x)\varphi(x)\mathrm{d}x&=\int_{\bR^d}\psi_0(t,-i\nabla)f(x)\varphi(x)\mathrm{d}x+\sum_{j=1}^{\infty}\int_{\bR^d}\psi_j(t,-i\nabla)f(x)\varphi(x)\mathrm{d}x\\
            &=\int_{\bR^d}f(x)\bar{\psi}(t,-i\nabla)\varphi(x)\mathrm{d}x,
        \end{align*}
        where $\bar{\psi}(t,-i\nabla)\varphi(x):=\cF^{-1}\left[\overline{\psi(t,\cdot)}\cF[\varphi]\right](x)$ and $\overline{\psi(t,\xi)}$ is the complex conjegate of $\psi(t,\xi)$.
        \item Then $\psi(t,-i\nabla):\Lambda^{\phi_{\gamma}}(\bR^d)\to \Lambda^{\phi}(\bR^d)$ is a bounded linear operator:
    $$
    \sup_{t\in[0,T]}\|\psi(t,-i\nabla)f\|_{\Lambda^{\phi}(\bR^d)}\leq N(d,\gamma,M,\Phi,\Psi)\|f\|_{\Lambda^{\phi_{\gamma}}(\bR^d)}.
    $$
    \end{enumerate} 
\end{thm}
\begin{proof}
$(i)$ Due to the dominated convergence theorem, the first equality holds.
If we use Fubini's theorem, Plancherel's theorem, and the dominated convergence theorem, then the second equality is also obtained 

$(ii)$  By $(i)$ and almost orthogonal property of Littlewood-Paley operators,
    \begin{align*}
        \Delta_j(\psi(t,-i\nabla)f)(x)&=(\Delta_{j-1}+\Delta_j+\Delta_{j+1})\Delta_j(\psi(t,-i\nabla)f)(x)\\
        &=\int_{\bR^d}(\Delta_{j-1}+\Delta_j+\Delta_{j+1})f(x-y)\bar{\psi}(t,-i\nabla)\Psi_j(y)\mathrm{d}y.
    \end{align*}
    To complete the proof, it suffices to prove that
    for $c>0$,
    \begin{equation}
    \label{23.12.27.18.38}
    \sup_{t\in[0,T]}\|\psi(t,-i\nabla)h_{c}\|_{L_1(\bR^d)}\leq N(d,\gamma,h,M)c^{\gamma},
    \end{equation}
    where $h_c(x):=c^dh(cx)$ and $h\in\mathcal{S}(\mathbb{R}^d)$.
    Indeed, if we have \eqref{23.12.27.18.38}, then
    \begin{align*}
        \phi(2^j)\|\Delta_j(\psi(t,-i\nabla)f)\|_{L_{\infty}(\bR^d)}&\leq N\phi(2^j)\|\psi(t,-i\nabla)\Psi_j\|_{L_1(\bR^d)}\sum_{k=-1}^{1}\|\Delta_{j+k}f\|_{L_{\infty}(\bR^d)}\\
        &\leq N\phi(2^j)2^{j\gamma}\sum_{k=-1}^{1}\|\Delta_{j+k}f\|_{L_{\infty}(\bR^d)}\\
        &\leq N(d,\gamma,M,\Psi)\|f\|_{\Lambda^{\phi_{\gamma}}(\bR^d)}.
    \end{align*}
    Similarly, we also have
    $$
    \|S_0(\psi(t,-i\nabla)f)\|_{L_{\infty}(\bR^d)}\leq N(d,\gamma,M,\Phi,\Psi)\|f\|_{\Lambda^{\phi_{\gamma}}(\bR^d)}.
    $$
    
    We end the proof by proving \eqref{23.12.27.18.38}.
    First, we decompose $\|\psi(t,-i\nabla)h_{c}\|_{L_1(\bR^d)}$ into two parts;
    $$
    \|\psi(t,-i\nabla)h_{c}\|_{L_1(\bR^d)}=\int_{|x|\leq c^{-1}}\cdots\mathrm{d}x+\int_{|x|>c^{-1}}\cdots\mathrm{d}x=:I_{c,1}(t)+I_{c,2}(t).
    $$
    By the assumptions on $\psi$ and the change of variable formula, we have
    \begin{equation}
    \label{24.11.21.13.26}
    \begin{aligned}
    |I_{c,1}(t)|&\leq Nc^{-d}\|\psi(t,-i\nabla)h_{c}\|_{L_{\infty}(\bR^d)}\\
    &\leq Nc^{-d}\int_{\bR^d}|\xi|^{\gamma}|\cF[h](c^{-1}\xi)|\mathrm{d}\xi\\
    &= Nc^{\gamma}\int_{\mathbb{R}^d}|\xi|^{\gamma}\mathcal{F}[h](\xi)\mathrm{d}\xi=N(d,\gamma,h,M)c^{\gamma}.
    \end{aligned}
    \end{equation}
    Since $h\in\mathcal{S}(\mathbb{R}^d)$, the integration $\int_{\mathbb{R}^d}|\xi|^{\gamma}\mathcal{F}[h](\xi)\mathrm{d}\xi$ in \eqref{24.11.21.13.26} is finite.
    For $I_{c,2}$, we use the Cauchy-Schwartz inequality and Plancherel's theorem;
    \begin{align*}
        |I_{c,2}(t)|&\leq \left(\int_{|x|>c^{-1}}|x|^{-2d_2}\mathrm{d}x\right)^{1/2}\left(\int_{|x|>c^{-1}}|x|^{2d_2}|\psi(t,-i\nabla)h_c(x)|^2\mathrm{d}x\right)^{1/2}\\
        &\leq N(d,\gamma,h,M)c^{d_2-d/2}\left(\int_{\bR^d}\left|D_{\xi}^{d_2}\left(\psi(t,\cdot)\cF[h](c^{-1}\cdot)\right)(\xi)\right|^2\mathrm{d}\xi\right)^{1/2}\leq Nc^{\gamma},
    \end{align*}
    where $d_2=\lfloor d/2\rfloor+1$.
    The theorem is proved.
\end{proof}

\subsection{Notion of solutions and their properties}
In this subsection, we discuss the notion of solutions to
\begin{equation}
    \label{23.11.30.14.23}
        \begin{cases}
        \partial_tu(t,x)=\psi(t,-i\nabla)u(t,x)+f(t,x),\quad (t,x)\in(0,T)\times\bR^d,\\
        u(0,x)=u_0(x).
    \end{cases}
    \end{equation}
It turns out that
\begin{align*}
p(t,s,x):=1_{0 \leq s< t} \cdot \frac{1}{(2\pi)^{d}}\int_{\bR^d} \exp\left(\int_{s}^t\psi(r,\xi)\mathrm{d}r\right)\mathrm{e}^{ix\cdot\xi}\mathrm{d}\xi+1_{t=s}\delta_0(x).
\end{align*}
is the \emph{fundamental solution} of \eqref{23.11.30.14.23}, \textit{i.e.}, $p(\cdot,s,\cdot)$ satisfies
\begin{equation*}
        \begin{cases}
        \partial_tp(t,s,x)=\psi(t,-i\nabla)p(t,s,x),\quad (t,x)\in(s,\infty)\times\bR^d,\\
        p(s,s,x)=\delta_0(x).
    \end{cases}
\end{equation*}
Here $\delta_0$ is the centered Dirac delta distribution.
The sharp $L_p$-estimates for the fundamental solution and their proof were presented in \cite[Theorem 6.1]{Choi_Kim2022} and \cite[Theorem 5.1, Corollary 5.3]{Choi_Kim_Lee2023}, thus, we present the following theorem without proofs.
\begin{thm}
\label{FSetimate}
    Let $\psi\in\mathbb{S}(\gamma,\kappa,M)$.
    Then for $m\in \{0,1\}$, and $\delta\in(0,1)$, there exists a positive constant $N=N(d,\delta,\gamma,\kappa,M,m)$ such that for all $t>s\geq0$,    \begin{equation*}
            \begin{aligned}
                &\left\|\partial_t^m\Delta_jp(t,s,\cdot)\right\|_{L_1(\bR^d)}\leq N\mathrm{e}^{-\kappa|t-s|2^{j\gamma}\times\frac{(1-\delta)}{2^{\gamma}}}2^{jm\gamma},\quad \forall j\in\bZ,\\
                &\left\|\partial_t^mS_0
                p(t,s,\cdot)\right\|_{L_1(\bR^d)}\leq N,\\
                &\left\|\partial_t^mp(t,s,\cdot)\right\|_{L_1(\bR^d)}\leq N(t-s)^{-m},\\
                &\left\||\cdot|^{d_2}\partial_t^mp(t,s,\cdot)\right\|_{L_2(\bR^d)}\leq N(t-s)^{-m+\frac{1}{\gamma}\left(d_2-\frac{d}{2}\right)},
            \end{aligned}
        \end{equation*}
        where $d_2:=\lfloor d/2\rfloor+1$.
\end{thm}
In general, the symbol $\psi(t,\xi)$ is complex-valued, thus
$$
\int_{\bR^d}p(t,s,y)\mathrm{d}y=\exp\left(\int_s^t\psi(r,0)\mathrm{d}r\right)\neq 1.
$$
However,
    \begin{equation}
\label{23.12.03.14.42}
    \lim_{t\downarrow s}\int_{\bR^d}p(t,s,y)\mathrm{d}y=1.
    \end{equation}

Now, we provide the concept of solutions.
\begin{defn}[Solutions]
\label{24.01.14.12.42}
    Let $p\in(1,\infty]$, $w\in A_p(\bR)$, $\phi\in\cI_o(0,L)$, $\psi\in\bS(\gamma,\kappa,M)$, $u_0\in \Lambda_p^{\phi_{\gamma,p,w}}(\bR^d)$ and $f\in L_p((0,T),w\,\mathrm{d}t;\Lambda^{\phi}(\bR^d))$.
    \begin{enumerate}[(i)]
        \item We say that
        $$
        u(t,x):=\int_{\bR^d}p(t,0,x-y)u_0(y)\mathrm{d}y+\int_0^t\int_{\bR^d}p(t,s,x-y)f(s,y)\mathrm{d}y\mathrm{d}s
        $$
        is the \emph{mild solution} of \eqref{23.11.30.14.23}.
        \item We say that $u\in L_p((0,T),w\,\mathrm{d}t;L_{\infty}(\bR^d))$ is a \emph{weak solution} of \eqref{23.11.30.14.23} if for $\varphi\in\cS(\bR^d)$,
        $$
        \int_{\bR^d}(u(t,x)-u_0(x))\varphi(x)\mathrm{d}x=\int_0^t\int_{\bR^d}\left(u(s,x)\bar{\psi}(s,-i\nabla)\varphi(x)+f(s,x)\varphi(x)\right)\mathrm{d}x\mathrm{d}s
        $$
        for all $t\in[0,T)$.
        Here $\bar{\psi}(s,-i\nabla)$ is a pseudo-differential operator with a symbol $\overline{\psi(s,\xi)}$.
        \item We say that $u\in L_p((0,T),w\,\mathrm{d}t;\Lambda^{\phi_{\gamma}}(\bR^d))$ is a \emph{strong solution} of \eqref{23.11.30.14.23} if
        $$
        u(t,x)=u_0(x)+\int_0^t(\psi(s,-i\nabla)u(s,x)+f(s,x))\mathrm{d}s
        $$
        for all $(t,x)\in[0,T)\times\bR^d$.
        Here $\psi(s,-i\nabla)u(s,x)$ is a function defined in \eqref{23.11.30.16.33}.
    \end{enumerate}
\end{defn}

\begin{rem}
\label{23.12.05.12.22}
    Due to Theorem \ref{23.11.14.14.54}, if $u$ is a strong solution to \eqref{23.11.30.14.23}, then $u$ is a weak solution also.
\end{rem}
For the mild solution of \eqref{23.11.30.14.23}, the following regularity estimate holds, and this plays a crucial role in the proof of Theorem \ref{23.12.05.09.53}.
\begin{thm}
\label{23.11.26.16.09}
    Let $\phi\in\cI_o(0,L)$, $p\in[1,\infty]$, $w\in A_p(\bR^d)$ and $\psi\in \mathbb{S}(\gamma,\kappa,M)$.
    Let also
\begin{equation*}
\begin{gathered}
\mathcal{G}f(t,x):= \int_0^t\int_{\bR^d}p(t,s,x-y) f(s,y)\mathrm{d}y\mathrm{d}s,\\
\mathcal{G}_0u_0(t,x):=\int_{\bR^d}p(t,0,x-y)u_0(y)\mathrm{d}y.
\end{gathered}
\end{equation*}

$(i)$ If $p\in(1,\infty]$ and $f\in {L_{p}((0,T),w\,\mathrm{d}t;\Lambda^{\phi}(\bR^d))}$, then $\mathcal{G}f\in {L_{p}((0,T),w\,\mathrm{d}t;\Lambda^{\phi_{\gamma}}(\bR^d))}$;
$$
\|\mathcal{G}f\|_{L_{p}((0,T),w\,\mathrm{d}t;\Lambda^{\phi_{\gamma}}(\bR^d))}\leq N\|f\|_{{L_{p}((0,T),w\,\mathrm{d}t;\Lambda^{\phi}(\bR^d))}},
$$
where $N=N(d,\delta,\gamma,\kappa,M,p,T,[w]_{A_p(\bR)})$.

$(ii)$ If $f\in {L_{1}((0,T),w\,\mathrm{d}t;\Lambda^{\phi}(\bR^d))}$, then $\mathcal{G}f\in L_{1,\infty}((0,T),w\,\mathrm{d}t;\Lambda^{\phi_{\gamma}}(\bR^d))$;
$$
\sup_{\lambda>0}\lambda w(\{t\in(0,T):\|\mathcal{G}f(t,\cdot)\|_{\Lambda^{\phi_{\gamma}}(\bR^d)}\geq\lambda\})\leq N\|f\|_{{L_{1}((0,T),w\,\mathrm{d}t;\Lambda^{\phi}(\bR^d))}},
$$
where $N=N(d,\delta,\gamma,\kappa,M,T,[w]_{A_1(\bR)})$.

$(iii)$ If $p\in(0,\infty]$, $w\in A_{\infty}(\bR)$ and $u_0\in\Lambda^{{\phi_{\gamma,p,w}}}(\bR^d)$, then
$$
\|\mathcal{G}_0u_0\|_{L_p((0,T),w\,\mathrm{d}t;\Lambda^{\phi_{\gamma}}(\bR^d))}\leq \|u_0\|_{\Lambda^{{\phi_{\gamma,p,w}}}(\bR^d)},
$$
where $N=N(d,\delta,\gamma,\kappa,M,p,T,[w]_{A_p(\bR)})$.
\end{thm}
The proof of Theorem \ref{23.11.26.16.09} will be provided in Section \ref{24.01.15.22.49}.

\subsection{Proof of Theorem \ref{23.12.05.09.53}}
In this subsection, we claim that the mild solution of \eqref{23.11.30.14.23} is a unique strong solution of \eqref{23.11.30.14.23}.

    \emph{Uniqueness.}
    Let $u$ be a strong solution to
    $$
    \begin{cases}
        \partial_tu(t,x)=\psi(t,-i\nabla)u(t,x),\quad (t,x)\in(0,T)\times\bR^d,\\
        u(0,x)=0.
    \end{cases}
    $$
    By Remark \ref{23.12.05.12.22} and
    the almost orthogonality of $\Delta_j$,
    \begin{align*}
    \Delta_ju(t,x)&=\int_0^t\int_{\bR^d}u(s,y)\bar{\psi}(s,-i\nabla)\Psi_j(x-y)\mathrm{d}y\mathrm{d}s\\
    &=\int_0^t\int_{\bR^d}\Delta_ju(s,y)\bar{\psi}(s,-i\nabla)(\Psi_{j-1}+\Psi_j+\Psi_{j+1})(x-y)\mathrm{d}y\mathrm{d}s.
    \end{align*}
    For every $j\in\bN$, by Minkowski's inequality and \eqref{23.12.27.18.38},
    \begin{align*}
        \|\Delta_ju(t,\cdot)\|_{L_{\infty}(\bR^d)}&\leq\sum_{k=-1}^{1}\int_0^t\|\bar{\psi}(s,-i\nabla)\Psi_{j+k}\|_{L_1(\bR^d)}\|\Delta_ju(s,\cdot)\|_{L_{\infty}(\bR^d)}\mathrm{d}s\\
        &\leq N2^{j\gamma}\int_0^t\|\Delta_ju(s,\cdot)\|_{L_{\infty}(\bR^d)}\mathrm{d}s.
    \end{align*} 
Due to Gr\"onwall's inequality, for all $t\in[0,T)$, $
\|\Delta_ju(t,\cdot)\|_{L_{\infty}(\bR^d)}=0$.
Similarly, we also have for all $t\in[0,T)$, $\|S_0u(t,\cdot)\|_{L_{\infty}(\bR^d)}=0$.
This implies that for $\varphi\in\cS(\bR^d)$
\begin{align*}
\left|\int_{\bR^d}u(t,x)\varphi(x)\mathrm{d}x\right|&=\left|\int_{\bR^d}u(t,x)\left(S_0\varphi(x)+\sum_{j=1}^{\infty}\Delta_j\varphi(x)\right)\mathrm{d}x\right|\\
&\leq \left|\int_{\bR^d}u(t,x)S_0\varphi(x)\mathrm{d}x\right|+\sum_{j=1}^{\infty}\left|\int_{\bR^d}u(t,x)\Delta_j\varphi(x)\mathrm{d}x\right|\\
&= \left|\int_{\bR^d}S_0u(t,x)\varphi(x)\mathrm{d}x\right|+\sum_{j=1}^{\infty}\left|\int_{\bR^d}\Delta_ju(t,x)\varphi(x)\mathrm{d}x\right|\\
&\leq \|\varphi\|_{L_1(\bR^d)}\left(\|S_0u(t,\cdot)\|_{L_{\infty}(\bR^d)}+\sum_{j=1}^{\infty}\|\Delta_ju(t,\cdot)\|_{L_{\infty}(\bR^d)}\right)=0
\end{align*}
for all $t\in[0,T)$.
By the fundamental lemma of the calculus of variations, $u=0$.
The uniqueness is proved.

\emph{Existence and estimates.}
Let $u$ be the mild solution of \eqref{23.11.30.14.23}.
By the definition of $\psi(s,-i\nabla)$,
\begin{align*}
    \int_0^t\psi(s,-i\nabla)u(s,x)\mathrm{d}s&=\int_0^t\int_{\bR^d}\psi(s,-i\nabla)p(s,0,x-y)u_0(y)\mathrm{d}y\mathrm{d}s\\
    &\quad+\int_0^t\int_0^s\int_{\bR^d}\psi(s,-i\nabla)p(s,r,x-y)f(r,y)\mathrm{d}y\mathrm{d}r\mathrm{d}s.
\end{align*}
Since $p$ is the fundamental solution, by the Fubini's theorem,
\begin{align*}
    &\int_0^t\int_{\bR^d}\psi(s,-i\nabla)p(s,0,x-y)u_0(y)\mathrm{d}y\mathrm{d}s+\int_0^t\int_0^s\int_{\bR^d}\psi(s,-i\nabla)p(s,r,x-y)f(r,y)\mathrm{d}y\mathrm{d}r\mathrm{d}s\\
    &=\int_0^t\int_{\bR^d}\partial_sp(s,0,x-y)u_0(y)\mathrm{d}y\mathrm{d}s+\int_0^t\int_0^s\int_{\bR^d}\partial_sp(s,r,x-y)f(r,y)\mathrm{d}y\mathrm{d}r\mathrm{d}s\\
    &=\int_0^t\partial_s\left(\int_{\bR^d}p(s,0,x-y)u_0(y)\mathrm{d}y\right)\mathrm{d}s+\int_0^t\int_r^t\partial_s\left(\int_{\bR^d}p(s,r,x-y)f(r,y)\mathrm{d}y\right)\mathrm{d}s\mathrm{d}r\\
    &=u(t,x)-\lim_{s\downarrow0}\int_{\bR^d}p(s,0,x-y)u_0(y)\mathrm{d}y-\lim_{s\downarrow r}\int_0^t\int_{\bR^d}p(s,r,x-y)f(r,y)\mathrm{d}y\mathrm{d}r.
\end{align*}
If we have
\begin{equation}
\label{23.12.03.14.20}
    \lim_{t\downarrow r}\int_{\bR^d}p(t,r,x-y)h(y)\mathrm{d}y=h(x)
\end{equation}
for $h\in \Lambda^{\phi}(\bR^d)$, then by the fundamental theorem of calculus, the existence is proved.

Since $h\in\Lambda^{\phi}(\bR^d)$, for given $\varepsilon>0$, there exists $\delta>0$ such that
$$
|x-y|<\delta\implies|h(x)-h(y)|<\varepsilon.
$$
This implies that
\begin{align*}
    &\left|\int_{\bR^d}p(t,r,x-y)h(y)\mathrm{d}y-h(x)\right|\\
    &\leq\left|\int_{|x-y|<\delta}p(t,r,x-y)h(y)\mathrm{d}y-h(x)\right|+\left|\int_{|x-y|\geq\delta}p(t,r,x-y)h(y)\mathrm{d}y\right|\\
    &\leq \int_{|x-y|<\delta}|p(t,r,x-y)||h(y)-h(x)|\mathrm{d}y+\left|\left(\int_{|x-y|<\delta}p(t,r,x-y)\mathrm{d}y-1\right)h(x)\right|\\
    &\quad +\|h\|_{L_{\infty}(\bR^d)}\int_{|x-y|\geq\delta}|p(t,r,x-y)|\mathrm{d}y\\
    &\leq \varepsilon\int_{|x-y|<\delta}|p(t,r,x-y)|\mathrm{d}y+\|h\|_{L_{\infty}(\bR^d)}\left|\int_{|x-y|<\delta}p(t,r,x-y)\mathrm{d}y-1\right|\\
    &\quad +\|h\|_{L_{\infty}(\bR^d)}\int_{|x-y|\geq\delta}|p(t,r,x-y)|\mathrm{d}y.
\end{align*}
By Theorem \ref{FSetimate},
\begin{align*}
    \int_{|x-y|\geq\delta}|p(t,r,x-y)|\mathrm{d}y&\leq \left(\int_{|y|\geq\delta}|y|^{-2d_2}\mathrm{d}y\right)^{1/2}\left(\int_{|y|\geq\delta_2}|y|^{2d_2}|p(t,r,y)|^2\mathrm{d}y\right)^{1/2}\\
    &\leq N\delta^{-\left(d_2-\frac{d}{2}\right)}(t-r)^{\frac{1}{\gamma}\left(d_2-\frac{d}{2}\right)},
\end{align*}
thus,
\begin{equation}
\label{24.01.25.17.34}
    \lim_{t\downarrow r}\int_{|x-y|\geq\delta}|p(t,r,x-y)|\mathrm{d}y=0.
\end{equation}
Since
$$
\left|\int_{|x-y|<\delta}p(t,r,x-y)\mathrm{d}y-1\right|\leq \left|\int_{\bR^d}p(t,r,x-y)\mathrm{d}y-1\right|+\int_{|x-y|\geq\delta}|p(t,r,x-y)|\mathrm{d}y,
$$
due to \eqref{23.12.03.14.42}, \eqref{24.01.25.17.34} and Theorem \ref{FSetimate},
$$
\limsup_{t\downarrow r}\left|\int_{\bR^d}p(t,r,x-y)h(y)\mathrm{d}y-h(x)\right|\leq N\varepsilon.
$$
Since $\varepsilon>0$ is an arbitrary positive constant, we have \eqref{23.12.03.14.20}.
Estimates \eqref{23.12.05.12.28} are obtained from Theorems \ref{23.11.14.14.54} and  \ref{23.11.26.16.09}.
The theorem is proved.\hfill $\square$

\subsection{Proof of Theorem \ref{23.11.26.16.09}}
\label{24.01.15.22.49}
First, we prove Theorem \ref{23.11.26.16.09}-$(i)$ and $(ii)$.
Recall that
$$
\mathcal{G}f(t,x):=\int_0^t\int_{\bR^d}p(t,s,x-y)f(s,y)\mathrm{d}y\mathrm{d}s,\quad \mathcal{G}_0u_0(t,x):=\int_{\bR^d}p(t,0,x-y)u_0(y)\mathrm{d}y.
$$
\begin{lem}
\label{23.12.05.14.49}
    There exists a constant $N=N(T)$ such that
    $$
    \|\cG f(t,\cdot)\|_{\Lambda^{\phi_{\gamma}}(\bR^d)}\leq N\cM(\|f(*,\cdot)\|_{\Lambda^{\phi}(\bR^d)}1_{(0,T)}(*))(t).
    $$
\end{lem}
\begin{proof}
    By Theorem \ref{FSetimate} and the almost orthogonality of the Littlewood-Paley projections,
    \begin{align*}
        &\|\cG f(t,\cdot)\|_{\Lambda^{\phi_{\gamma}}(\bR^d)}\\
        &\leq \int_{0}^t\|S_0(p(t,s,\cdot)\ast f(s,\cdot))\|_{L_{\infty}(\bR^d)}\mathrm{d}s+\sup_{j\in\bN}\phi(2^{j})2^{j\gamma}\int_{0}^{t}\|\Delta_j(p(t,s,\cdot)\ast f(s,\cdot))\|_{L_{\infty}(\bR^d)}\mathrm{d}s\\
        &\leq 
        \int_0^t\|(S_0+\Delta_1)p(t,s,\cdot)\|_{L_1(\bR^d)}\|f(s,\cdot)\|_{\Lambda^{\phi}(\bR^d)}\mathrm{d}s\\
        &\quad+\sup_{j\in\bN}2^{j\gamma}\int_{0}^{t}\|(\Delta_{j-1}+\Delta_{j}+\Delta_{j+1})p(t,s,\cdot)\|_{L_1(\bR^d)}\|f(s,\cdot)\|_{\Lambda^{\phi}(\bR^d)}\mathrm{d}s\\
        &\leq N\int_0^t\|f(s,\cdot)\|_{\Lambda^{\phi}(\bR^d)}\mathrm{d}s +N\sup_{j\in\bN}2^{j\gamma}\int_{0}^{t}\mathrm{e}^{-\kappa |t-s|2^{j\gamma}\times\frac{(1-\delta)}{2^{\gamma}}}\|f(s,\cdot)\|_{\Lambda^{\phi}(\bR^d)}\mathrm{d}s\\
        &\leq N(T)\cM(\|f(*,\cdot)\|_{\Lambda^{\phi}(\bR^d)}1_{(0,T)}(*))(t)+N\sup_{j\in\bN}\int_0^tk_j(t-s)\|f(s,\cdot)\|_{\Lambda^{\phi}(\bR^d)}\mathrm{d}s,
    \end{align*}
where $k_j(t):=2^{j\gamma}\mathrm{e}^{-\kappa|t|2^{j\gamma}\times\frac{(1-\delta)}{2^{\gamma}}}$.
Let $k_{j,n}(t):=k_j(t)\eta_n(t)$ where $\eta_n$ is a nonnegative infinitely differentiable function satisfying $\eta_n(t)=1$ for $0\leq t\leq n-1$ and $\eta_n(t)=0$ for $t\geq n$.
Then due to the integration by parts formula, for any nonnegative function $F$,
\begin{align*}
    \int_0^{\infty}k_{j,n}(s)F(s)\mathrm{d}s&=-\int_{0}^{n}k_{j,n}'(s)\left(\int_0^sF(r)\mathrm{d}r\right)\mathrm{d}s\\
    &\leq 2\cM F(0) \int_0^ns(-k_{j,n}'(s))\mathrm{d}s\\
    &=2\cM F(0)\int_0^nk_{j,n}(s)\mathrm{d}s.
\end{align*}
By the monotone convergence theorem,
\begin{equation}
\label{24.02.03.17.24}
    \int_0^{\infty}k_j(s)F(s)\mathrm{d}s\leq N \cM F(0).
\end{equation}
Since
$$
\sup_{j\in\bN}\int_0^{\infty}k_j(t)\mathrm{d}t<\infty,
$$
the constant $N$ in \eqref{24.02.03.17.24} is independent of $j$.
By putting $F(s):=\|f(t-s,\cdot)\|_{\Lambda^{\phi}(\bR^d)}1_{(0,T)}(t-s)$, we have
$$
\cM F(0)=\cM(\|f(*,\cdot)\|_{\Lambda^{\phi}}1_{(0,T)}(*))(t).
$$
The lemma is proved.
\end{proof}

\begin{proof}[Proof of Theorem \ref{23.11.26.16.09}]
$(i)$, $(ii)$
It is well known that (\textit{e.g.} \cite[Theorem 7.1.9]{grafakos2014classical})
\begin{equation}
\label{23.12.05.14.47}
    \|\mathcal{M}\|_{L_{1}(\bR,w\,\mathrm{d}t)\to L_{1,\infty}(\bR,w\,\mathrm{d}t)}+\|\mathcal{M}\|_{L_p(\bR,w\,\mathrm{d}t)\to L_p(\bR,w\,\mathrm{d}t)}1_{p\in(1,\infty)}\leq C(p,[w]_{A_p(\bR)}).
\end{equation}
It can be also easily checked that
\begin{equation}
\label{23.12.05.14.48}
    \|\cM\|_{L_{\infty}(\bR)\to L_{\infty}(\bR)}\leq 1.
\end{equation}
By combining \eqref{23.12.05.14.47}, \eqref{23.12.05.14.48} and Lemma \ref{23.12.05.14.49}, $(i)$ and $(ii)$ are proved.

$(iii)$ By Theorem \ref{23.11.14.14.54} and the almost orthogonality of the Littlewood-Paley projections,
    \begin{align*}
        &\|\cG_0 u_0(t,\cdot)\|_{\Lambda^{\phi_{\gamma}}(\bR^d)}\\
        &= \|S_0(p(t,0\cdot)\ast u_0)\|_{L_{\infty}(\bR^d)}+\sup_{j\in\bN}\phi(2^{j})2^{j\gamma}\|\Delta_j(p(t,0,\cdot)\ast u_0)\|_{L_{\infty}(\bR^d)}\\
        &\leq N\|u_0\|_{\Lambda_p^{\phi_{\gamma,w,p}}(\bR^d)}+N\sup_{j\in\bZ}\phi(2^j)2^{j\gamma}\|(\Delta_{j-1}+\Delta_{j}+\Delta_{j+1})p(t,0,\cdot)\|_{L_1(\bR^d)}\|\Delta_ju_0\|_{L_{\infty}(\bR^d)}\\
        &\leq N\|u_0\|_{\Lambda_p^{\phi_{\gamma,w,p}}(\bR^d)}+N\sup_{j\in\bZ}\phi(2^j)2^{j\gamma}\mathrm{e}^{-N(\delta,\gamma,\kappa) t2^{j\gamma}}\|\Delta_ju_0\|_{L_{\infty}(\bR^d)}.
    \end{align*}
    If $p=\infty$, then $\phi_{\gamma,w,p}(\lambda)=\phi_{\gamma}(\lambda)$ and
    $$
    \sup_{j\in\bZ}\phi(2^j)2^{j\gamma}\mathrm{e}^{-N t2^{j\gamma}}\|\Delta_ju_0\|_{L_{\infty}(\bR^d)}\leq N(\delta,\gamma,\kappa)\|u_0\|_{\Lambda^{\phi_{\gamma}}(\bR^d)}.
    $$
    We complete the proof by showing if $p\in(0,\infty)$, then
    \begin{equation}
    \label{23.12.05.14.58}
        \int_0^{\infty}\left(\sup_{j\in\bZ}\phi(2^j)2^{j\gamma}\mathrm{e}^{-N t2^{j\gamma}}\|\Delta_ju_0\|_{L_{\infty}(\bR^d)}\right)^{p}w(t)\mathrm{d}t\leq N\|u_0\|_{\Lambda_{p}^{\phi_{\gamma,w,p}}(\bR^d)}^p.
    \end{equation}
    It can be easily checked that
    \begin{align*}
        &\int_0^{\infty}\left(\sup_{j\in\bZ}\phi(2^j)2^{j\gamma}\mathrm{e}^{-N t2^{j\gamma}}\|\Delta_ju_0\|_{L_{\infty}(\bR^d)}\right)^{p}w(t)\mathrm{d}t\\
        &\leq \int_0^{\infty}\sup_{j\in\bZ}\phi(2^j)^p2^{jp\gamma}\mathrm{e}^{-N t2^{j\gamma}}\|\Delta_ju_0\|_{L_{\infty}(\bR^d)}^pw(t)\mathrm{d}t\\
        &\leq \sum_{j\in\bZ}\phi(2^j)^p2^{jp\gamma}\|\Delta_ju_0\|_{L_{\infty}(\bR^d)}^p\int_0^{\infty}\mathrm{e}^{-N t2^{j\gamma}}w(t)\mathrm{d}t\\
        &=\sum_{j\in\bZ}\phi(2^j)^p2^{jp\gamma}\cL[w](N2^{j\gamma})\|\Delta_ju_0\|_{L_{\infty}(\bR^d)}^p,
    \end{align*}
where $N=N(\delta,\gamma,\kappa,p)$ and $\mathcal{L}[w](\lambda):=\int_0^{\infty}\mathrm{e}^{-\lambda t}w(t)\mathrm{d}t$.
By \cite[Lemma 3.1]{Choi_Kim_Lee2023}, for any $w\in A_{\infty}(\bR)$,
    \begin{equation*}
        \mathcal{L}[w](N2^{j\gamma})\simeq \int_0^{2^{-j\gamma}}w(t)\mathrm{d}t=W(2^{-j\gamma}).
    \end{equation*}
The theorem is proved.   
\end{proof}

\mysection{Embedding theorems: Proofs of Theorem \ref{24.02.03.17.37}}
\label{24.02.03.18.24}
In this section, we focus on proving Theorem \ref{24.02.03.17.37}.
To facilitate this, we introduce the following lemma, which characterizes the generalized real interpolation space.
This lemma plays a crucial role in the proof of Theorem \ref{24.02.03.17.37}.
For further properties and a detailed exploration of the generalized real interpolation space, readers are referred to \cite{CLSW2023}.

\begin{lem}
\label{24.01.15.23.12}
    Let $\gamma\in(0,\infty)$, $\phi\in\cI_o(0,L)$, $p\in(1,\infty)$ and $w\in A_p(\bR)$.
    Then 
    $$
    (\Lambda^{\phi_{\gamma}}(\bR^d),\Lambda^{\phi}(\bR^d))_{W^{1/p},p}=\Lambda_p^{\phi_{\gamma,p,w}}(\bR^d),
    $$
    where $(\Lambda^{\phi_{\gamma}}(\bR^d),\Lambda^{\phi}(\bR^d))_{W^{1/p},p}$ is the real interpolation space between $\Lambda^{\phi_{\gamma}}(\bR^d)$ and $\Lambda^{\phi}(\bR^d)$.
    The norm of $(\Lambda^{\phi_{\gamma}}(\bR^d),\Lambda^{\phi}(\bR^d))_{W^{1/p},p}$ is defined by
    $$
    \|f\|_{(\Lambda^{\phi_{\gamma}}(\bR^d),\Lambda^{\phi}(\bR^d))_{W^{1/p},p}}:=\int_{0}^{\infty}W(t^{-1})K(t,f;\Lambda^{\phi_{\gamma}}(\bR^d),\Lambda^{\phi}(\bR^d))^p\frac{\mathrm{d}t}{t},
    $$
    and $W(t):=\int_0^tw(s)\mathrm{d}s$.
    Here, the $K$-functional is defined by
    $$
    K(t,f;\Lambda^{\phi_{\gamma}}(\bR^d),\Lambda^{\phi}(\bR^d)):=\inf(\|h_0\|_{\Lambda^{\phi_{\gamma}}(\bR^d)}+t\|h_1\|_{\Lambda^{\phi}(\bR^d))}),
    $$
    where the infimum is taken over all pairs $(h_0,h_1)$ satisfying
    $$f(x)=h_0(x)+h_1(x),\quad h_0\in\Lambda^{\phi_{\gamma}}(\bR^d),\quad h_1\in\Lambda^{\phi}(\bR^d).
    $$
\end{lem}
\begin{proof}
    Let $A$ be a Banach space and $\ell_{\infty}^{\phi}(A)$ be a space of $A$-valued sequences $a=\{a_{n}\}_{n\in\bZ}$ satisfying
    $$
    \|a\|_{\ell_{\infty}^{\phi}(A)}:=\sup_{n\in\bZ}\phi(2^n)\|a_n\|_{A}<\infty.
    $$
    Due to \cite[Theorem 2.4.2/(a)]{triebel1978interpolation} with \cite[Theorem 1]{kurtz1979results},
    it suffices to prove only
    $$
    (\ell_{\infty}^{\phi_{\gamma}}(A),\ell_{\infty}^{\phi}(A))_{W^{1/p},p}=\ell_{p}^{\phi_{\gamma,p,w}}(A).
    $$
    We follow the proof of \cite[Proposition A.4]{CLSW2023trace}.

    \emph{Step 1.} We first prove that
    $$
    (\ell_{\infty}^{\phi_{\gamma}}(A),\ell_{\infty}^{\phi}(A))_{W^{1/p},p}\subset\ell_{p}^{\phi_{\gamma,p,w}}(A).
    $$
    Note that for a $A$-valued sequence $a=\{a_n\}_{n\in\bZ}$,
    \begin{align*}
    K(t,a;\ell_{\infty}^{\phi_{\gamma}}(A),\ell_{\infty}^{\phi}(A))&=\inf_{a=b+c}(\|b\|_{\ell_{\infty}^{\phi_{\gamma}}(A)}+t\|c\|_{\ell_{\infty}^{\phi}(A)})\\
    &\simeq \sup_{n\in\bZ}\min(\phi(2^n)2^{n\gamma},t\phi(2^n))\|a_n\|_{A}.
    \end{align*}
    Hence
    \begin{equation*}\label{ineq-221012 1340}
	\begin{aligned}
		\| a \|_{(\ell_{\infty}^{\phi_{\gamma}}(A),\ell_{\infty}^{\phi}(A))_{W^{1/p}, p}}^p
		&= \int_0^\infty  W(t^{-1}) K(t, a;  \ell_{\infty}^{\phi_{\gamma}}(A),\ell_{\infty}^{\phi}(A))^p \frac{\mathrm{d}t}{t}\\
		&\simeq \int_0^\infty  W(t^{-1}) \left(\sup_{n\in\mathbb{Z}} \left(\min(\phi(2^n)2^{n\gamma},t\phi(2^n))\|a_n\|_{A} \right)\right)^p \frac{\mathrm{d}t}{t}\\
		&= \sum_{j\in\mathbb{Z}} \int_{2^{j\gamma}}^{2^{(j+1)\gamma}}  W(t^{-1}) \left(\sup_{n\in\mathbb{Z}} \left(\min(\phi(2^n)2^{n\gamma},t\phi(2^n))\|a_n\|_{A} \right)\right)^p
		\frac{\mathrm{d}t}{t}\\
		&\simeq \sum_{j\in\mathbb{Z}} W(2^{-j\gamma})   \sup_{n\in\mathbb{Z}} \left(\min(\phi(2^n)2^{n\gamma}, 2^{j\gamma}\phi(2^{n})) \|a_n\|_{A}\right)^p\\
		&\geq \sum_{j\in\mathbb{Z}} W(2^{-j\gamma}) \phi(2^j)^p2^{jp\gamma} \|a_j\|_{A}^p=\|a\|_{\ell_p^{\phi_{\gamma,p,w}}(A)}.
	\end{aligned}
	\end{equation*}

\emph{Step 2.} In this step, we prove that
\begin{align}\label{left_ineq}
		\ell_p^{\phi_{\gamma,p,w}}(A) \subset (\ell_1^{\phi_{\gamma}}(A), \ell_1^{\phi}(A))_{W^{1/p}, p}.
	\end{align} 
Note that 
	\begin{align*}
		K(t, a; \ell_1^{\phi_{\gamma}}(A), \ell_1^{\phi}(A)) \simeq \sum_{n\in\mathbb{Z}} \min(\phi(2^{n})2^{n\gamma}, t\phi(2^{n})) \|a_n\|_{A}.
	\end{align*}
 
Then for $q\in[1,\infty)$, we have
	\begin{equation}\label{230706542}
\begin{aligned}
  &\| a \|_{(\ell_1^{\phi_{\gamma}}(A), \ell_1^{\phi}(A))_{W^{1/p}, p}}^p=\int_0^\infty W(t^{-1})K(t, a;  \ell_1^{\phi_{\gamma}}(A), \ell_1^{\phi}(A))^p \frac{\mathrm{d}t}{t}\\
		&\simeq \sum_{j\in\mathbb{Z}} \int_{2^{j\gamma}}^{2^{(j+1)\gamma}} W(t^{-1})\left(  \sum_{n\in\mathbb{Z}} \min(\phi(2^{n})2^{n\gamma}, t\phi(2^{n})) \|a_n\|_{A}\right)^p \frac{\mathrm{d}t}{t}\\
		&\simeq \sum_{j\in\mathbb{Z}} W(2^{-j\gamma})   \left( \sum_{n\in\mathbb{Z}} \min(\phi(2^{n\gamma})2^{n\gamma}, \phi(2^{n})2^{j\gamma}) \|a_n\|_{A}\right)^p\\
		&\lesssim  \sum_{j\in\mathbb{Z}} W(2^{-j\gamma})  \left( \left( \sum_{n \leq j}\phi(2^n)2^{n\gamma} \|a_n\|_{A} \right)^p+  \left(\sum_{n>j} \phi(2^{n})2^{j\gamma} \|a_n\|_{A}\right)^p \right)\\
		&=: I + II.
	\end{aligned}
 \end{equation}
We handle terms $I$ and $II$ separately.

\emph{Step 2-1.} For $\chi>0$, by H\"older's inequality,
\begin{align*}
\left(\sum_{n\leq j}\phi(2^n)2^{n\gamma}\|a_n\|_{A}\right)^p&\leq \left(\sum_{n\leq j}2^{n\chi p'}\right)^{p-1}\left(\sum_{n\leq j}\phi(2^n)^p2^{n(\gamma-\chi) p}\|a_n\|_{A}^p\right)\\
&=N2^{j\chi p}\sum_{n\leq j}\phi(2^n)2^{n(\gamma-\chi)p}\|a_n\|_{A}^p.
\end{align*}
By Fubini's theorem,
\begin{align*}
    I\leq \sum_{n\in\bZ}\phi(2^n)2^{n(\gamma-\chi)p}\|a_n\|_{A}^p\sum_{n\leq j}W(2^{-j\gamma})2^{j\chi p}.
\end{align*}
Since there exists $\delta\in(0,1)$ such that
$$
W(2^{-j\gamma})\leq N2^{-(j-n)\delta\gamma}W(2^{-n\gamma}),\quad \forall j\geq n,
$$
if we take $\chi\in(0,\delta\gamma/p)$, then
$$
I\leq N\|a\|_{\ell_p^{\phi_{\gamma,p,w}}(A)}^p.
$$

\emph{Step 2-2.} For $\chi>0$, by H\"older's inequality,
\begin{align*}
\left(\sum_{n< j}\phi(2^n)\|a_n\|_{A}\right)^p&\leq \left(\sum_{n> j}2^{-n\chi p'}\right)^{p-1}\left(\sum_{n> j}\phi(2^n)^p2^{n\chi p}\|a_n\|_{A}^p\right)\\
&=N2^{-j\chi p}\sum_{n> j}\phi(2^n)^p2^{n\chi p}\|a_n\|_{A}^p.
\end{align*}
By Fubini's theorem,
\begin{align*}
    II\leq \sum_{n\in\bZ}\phi(2^n)^p2^{n\chi p}\|a_n\|_{A}^p\sum_{n> j}W(2^{-j\gamma})2^{j(\gamma-\chi)p}.
\end{align*}
Since there exists $\delta\in(0,1)$ such that
$$
W(2^{-j\gamma})\leq N2^{-(j-n)\gamma (p-\delta)}W(2^{-n\gamma}),\quad \forall j<n,
$$
if we take $\chi\in(0,\delta\gamma/p)$, then
$$
II\leq N\|a\|_{\ell_p^{\phi_{\gamma,p,w}}(A)}^p.
$$

Since $\ell_{1}^{\phi_{\gamma}}(A)\subset \ell_{\infty}^{\phi_{\gamma}}(A)$ and $\ell_{1}^{\phi}(A)\subset \ell_{\infty}^{\phi},(A)$ by combining the results in Step $1$ and $2$, we have
\begin{equation*}
	\begin{aligned}
		\ell_p^{\phi_{\gamma,p,w}}(A) 
		&\subseteq (\ell_1^{\phi_{\gamma}}(A), \ell_1^{\phi}(A))_{W^{1/p}, p}\\
		&\subseteq (\ell_\infty^{\phi_{\gamma}}(A), \ell_\infty^{\phi}(A))_{W^{1/p}, p}
		\subseteq \ell_p^{\phi_{\gamma,p,w}}(A).
	\end{aligned}
	\end{equation*}
The lemma is proved.
\end{proof}

We end this section by proving Theorem \ref{24.02.03.17.37}.
\begin{proof}[Proof of Theorem \ref{24.02.03.17.37}]

    $(i)$ The left inequality in \eqref{24.01.15.23.11} is easily obtained from
    \begin{align*}
    \|u(t,\cdot)\|_{\Lambda^{\phi_{\gamma,p,w}}(\bR^d)}&=\|S_0u(t,\cdot)\|_{L_{\infty}(\bR^d)}+\sup_{j\in\bN}\phi_{\gamma,p,w}(2^j)\|\Delta_ju(t,\cdot)\|_{L_{\infty}(\bR^d)}\\
    &\leq\|S_0u(t,\cdot)\|_{L_{\infty}(\bR^d)}+\left(\sum_{j\in\bN}\phi_{\gamma,p,w}(2^j)^p\|\Delta_ju(t,\cdot)\|_{L_{\infty}(\bR^d)}^p\right)^{1/p}\\
    &=\|u(t,\cdot)\|_{\Lambda_p^{\phi_{\gamma,p,w}}(\bR^d)}.
    \end{align*}
    For the right inequality in \eqref{24.01.15.23.11}, by Lemma \ref{24.01.15.23.12}, it is enough to show that
    $$
    \sup_{0\leq t\leq T}\|u(t,\cdot)\|_{(\Lambda^{\phi_{\gamma}}(\bR^d),\Lambda^{\phi}(\bR^d))_{W^{1/p},p}}\leq N\|u\|_{\mathbf{H}_{p,w}^{\phi,{\gamma}}((0,T)\times\bR^d)}.
    $$
    Let $u\in \mathbf{H}_{p,w}^{\phi,{\gamma}}((0,T)\times\bR^d)$ satisfying
    $$
    u(t,x)=u_0(x)+\int_0^tf(s,x)\mathrm{d}s
    $$
    for all $(t,x)\in[0,T)\times\bR^d$.
    Let also $\zeta\in C^{\infty}((0,\infty))$ be a nonnegative decreasing function satisfying $\zeta(s)=1$ for $0\leq s\leq T/2$ and $\zeta(s)=0$ for $s\geq T$.
    Denote
    $$
    \bar{u}(t,x):=\zeta(t)u(t,x),\quad \bar{f}(t,x):=\zeta(t) f(t,x)+\zeta'(t)u(t,x)
    $$
    Then for $0\leq s\leq t$ and $x\in\bR^d$,
    \begin{equation}
    \label{23.12.08.14.22}
        \bar{u}(t,x)=\bar{u}(s,x)+\int_s^t\bar{f}(r,x)\mathrm{d}r.
    \end{equation}
    Taking the Laplace transform in \eqref{23.12.08.14.22} with respect to $t\geq s$,
    \begin{equation}
    \label{23.12.08.14.36}
        \bar{u}(s,x)=\lambda\cL_s[\bar{u}(\cdot,x)](\lambda)-\cL_s[\bar{f}(\cdot,x)](\lambda),
    \end{equation}
    where $\cL_s[h](\lambda):=\int_s^{\infty}\mathrm{e}^{-\lambda t}h(r)\mathrm{d}r$.
    By \eqref{23.12.08.14.36},
    $$
    K(\lambda,\bar{u}(s,\cdot);\Lambda^{\phi_{\gamma}}(\bR^d),\Lambda^{\phi}(\bR^d))\leq \lambda \cL_0[\|\bar{u}(\cdot,*)\|_{\Lambda^{\phi_{\gamma}}(\bR^d)}](\lambda)+\lambda \cL_0[\|\bar{f}(\cdot,*)\|_{\Lambda^{\phi}(\bR^d)}](\lambda).
    $$
    Then it follows that
    \begin{align*}
        &\|\bar{u}(s,\cdot)\|_{(\Lambda^{\phi_{\gamma}}(\bR^d),\Lambda^{\phi}(\bR^d))_{W^{1/p},p}}^p\\
        &=:\int_0^{\infty}W(\lambda^{-1})K(\lambda,\bar{u}(s,\cdot);\Lambda^{\phi_{\gamma}}(\bR^d),\Lambda^{\phi}(\bR^d))^p\frac{\mathrm{d}\lambda}{\lambda}\\
        &\leq N(p)\int_0^{\infty}W(\lambda^{-1})|\lambda \cL_0[\|\bar{u}(\cdot,*)\|_{\Lambda^{\phi_{\gamma}}(\bR^d)}](\lambda)|^p+|\lambda \cL_0[\|\bar{f}(\cdot,*)\|_{\Lambda^{\phi}(\bR^d)}](\lambda)|^p\frac{\mathrm{d}\lambda}{\lambda}.
    \end{align*}
By \cite[Equation (4.9)]{CLSW2023trace}, we obtain
\begin{align*}
&\|\zeta(s)u(s,\cdot)\|_{(\Lambda^{\phi_{\gamma}}(\bR^d),\Lambda^{\phi}(\bR^d))_{W^{1/p},p}}\\
&\quad=\|\bar{u}(s,\cdot)\|_{(\Lambda^{\phi_{\gamma}}(\bR^d),\Lambda^{\phi}(\bR^d))_{W^{1/p},p}}\\
&\quad\leq N\left(\|\bar{u}\|_{L_p((0,\infty),w\,\mathrm{d}t;\Lambda^{\phi_{\gamma}}(\bR^d))}+\|\bar{f}\|_{L_p((0,\infty),w\,\mathrm{d}t;\Lambda^{\phi}(\bR^d))}\right)\\
&\quad\leq N\left(\|u\|_{L_p((0,T),w\,\mathrm{d}t;\Lambda^{\phi_{\gamma}}(\bR^d))}+\|f\|_{L_p((0,T),w\,\mathrm{d}t;\Lambda^{\phi}(\bR^d))}\right)\\
&\quad\leq N\|u\|_{\mathbf{H}_{p,w}^{\phi,{\gamma}}((0,T)\times\bR^d)},
\end{align*}
where $N=N(p,[w]_{A_p(\bR)})$.
Using the similar arguments, we also have
\begin{align*}
    \|(1-\zeta(s))u(s,\cdot)\|_{(\Lambda^{\phi_{\gamma}}(\bR^d),\Lambda^{\phi}(\bR^d))_{W^{1/p},p}}\leq N\|u\|_{\mathbf{H}_{p,w}^{\phi,{\gamma}}((0,T)\times\bR^d)}.
\end{align*}

    $(ii)$ Let $u\in \mathbf{H}_{p,w}^{\phi,{\gamma}}((0,T)\times\bR^d)$ satisfying
    $$
    u(t,x)=u_0(x)+\int_0^tf(s,x)\mathrm{d}s
    $$
    for all $(t,x)\in[0,T)\times\bR^d$.
    By H\"older's inequality,
    $$
    \|u(t,\cdot)-u(s,\cdot)\|_{\Lambda^{\phi}(\bR^d)}\leq \int_s^t\|f(r,\cdot)\|_{\Lambda^{\phi}(\bR^d)}\mathrm{d}r\leq \widetilde{W}(t,s)^{1-\frac{1}{p}}\|f\|_{L_p((0,T),w\,\mathrm{d}t;\Lambda^{\phi}(\bR^d))},
    $$
    and this certainly implies that
    $$
    \sup_{0\leq s< t<T}\frac{\|u(t,\cdot)-u(s,\cdot)\|_{\Lambda^{\phi}(\bR^d)}}{\widetilde{W}(t,s)^{1-1/p}}\leq N\|u\|_{\mathbf{H}_{p,w}^{\phi,{\gamma}}((0,T)\times\bR^d)}.
    $$
    By Theorem \ref{23.12.05.09.53}-$(ii)$, we have
    \begin{align*}
        &\sup_{(t,x)\in[0,T]\times\bR^d}|u(t,x)|+\sup_{0\leq t\leq T}\sup_{x\in\bR^d,|h|\leq 1}\frac{|\mathcal{D}_h^{L_0}u(t,x)|}{\phi_{\gamma,p,w}(|h|^{-1})^{-1}}+\sup_{0\leq s< t\leq T}\sup_{x\in\bR^d}\frac{|u(t,x)-u(s,x)|}{\widetilde{W}(t,s)^{1-1/p}}\\
        &\leq N\|u\|_{\mathbf{H}_{p,w}^{\phi,{\gamma}}((0,T)\times\bR^d)}.
    \end{align*}
    This implies that
    \begin{align*}
    &\frac{|\mathcal{D}^{L_0}_{(h_1/L_0,h_2)}u(s,x)|}{\widetilde{W}(s+h_1,s)^{1-1/p}+\phi_{\gamma,p,w}(|h_2|^{-1})^{-1}}\\
    &\leq \sum_{i=1}^{L_0}\frac{|u(s+ih_1/L_0,\cdot)-u(s+(i-1)h_1/L_0,\cdot)|_{C(\bR^d)}}{\widetilde{W}(s+h_1,s)^{1-1/p}}+\frac{|\mathcal{D}^{L_0}_{h_2}u(\cdot,x)|_{C((0,T))}}{\phi_{\gamma,p,w}(|h_2|^{-1})^{-1}}\\
    &\leq  \sum_{i=1}^{L_0}\frac{|u(s+ih_1/L_0,\cdot)-u(s+(i-1)h_1/L_0,\cdot)|_{C(\bR^d)}}{\widetilde{W}(s+ih_1/L_0,s+(i-1)h_1/L_0)^{1-1/p}}+\frac{|\mathcal{D}^{L_0}_{h_2}u(\cdot,x)|_{C((0,T))}}{\phi_{\gamma,p,w}(|h_2|^{-1})^{-1}}\\
    &\leq N\|u\|_{\mathbf{H}_{p,w}^{\phi,{\gamma}}((0,T)\times\bR^d)}.
    \end{align*}
    The theorem is proved.
\end{proof}

\vspace{1cm}
\textbf{Declarations of interest.}
Declarations of interest: none

\textbf{Data Availability.}
Data sharing not applicable to this article as no datasets were generated or analysed during the current study.

\textbf{Acknowledgements.}
The author has been supported by the National Research Foundation of Korea(NRF) grant funded by the Korea government(ME) (No.RS-2023-00237315).
The author would like to thank the referees for their careful reading and useful comments.
The author also would like to thank Professor Ildoo Kim, Dr. Jin Bong Lee, and Dr. Jinsol Seo for giving many helpful comments.

\appendix

\bibliographystyle{plain}

\end{document}